\documentclass[a4paper,reqno,11pt]{amsart}
\usepackage{amsmath}
\usepackage{amssymb}
\usepackage{cases}
\allowdisplaybreaks[4]

\newtheorem{thm}{Theorem}[section]
\newtheorem{prop}[thm]{Proposition}

\newtheorem{lem}[thm]{Lemma}

\newtheorem{cor}[thm]{Corollary}

\theoremstyle{definition}

\newtheorem{example}[thm]{Example}

\theoremstyle{remark}
\newtheorem{remark}[thm]{Remark}

\numberwithin{equation}{section}

\newcommand{\bQ}{\mathbb{Q}}
\newcommand{\bP}{\mathbb{P}}

\newcommand\OO{{\mathcal{O}}}

\newcommand{\roundup}[1]{\lceil{#1}\rceil}

\newcommand{\bC}{{\mathbb C}}
\newcommand{\bR}{{\mathbb R}}
\newcommand{\bN}{{\mathbb N}}

\newcommand\PP{{\mathcal{P}}}

\newcommand\III{{\mathcal{I}}}

\newcommand\mult{{\text{\rm mult}}}

\newcommand\vol{{\text{\rm vol}}}
\newcommand\lct{{\text{\rm lct}}}

\begin{document}

\title{Boundedness of $\bQ$-Fano varieties with degrees and alpha-invariants bounded from below}
\date{\today}
\author{Chen Jiang}
\address{Kavli IPMU (WPI), UTIAS, The University of Tokyo, Kashiwa, Chiba 277-8583, Japan.}
\email{chen.jiang@ipmu.jp}
\thanks{The author was supported by JSPS KAKENHI Grant Number JP16K17558 and World Premier International Research Center Initiative (WPI), MEXT, Japan.}

\begin{abstract}
We show that $\bQ$-Fano varieties of fixed dimension with anti-canonical degrees and alpha-invariants bounded from below form a bounded family. As a corollary, K-semistable $\bQ$-Fano varieties of fixed dimension with anti-canonical degrees bounded from below form a bounded family.
\end{abstract}
\maketitle

\pagestyle{myheadings} \markboth{\hfill   Chen Jiang
\hfill}{\hfill Boundedness of $\bQ$-Fano varieties with degrees and $\alpha$-invariants bounded from below\hfill}

\section{Introduction}
Throughout the article, we work over an algebraically closed field of characteristic zero. 
 A {\it $\bQ$-Fano variety} is defined to be a normal projective variety $X$ with at most klt singularities 
such that the anti-canonical divisor $-K_X$ is an ample $\bQ$-Cartier divisor.

When the base field is the
complex number field, an interesting problem for $\bQ$-Fano varieties is the existence of K\"ahler--Einstein metrics which  is related to K-(semi)stability of $\bQ$-Fano varieties.  It has been known that a {\it Fano manifold} $X$ (i.e., a smooth $\bQ$-Fano variety over $\bC$) admits
K\"ahler--Einstein metrics if and only if $X$ is {\it K-polystable} by the works
\cite{DT, Tia97, Don02, Don05, CT08, Sto09, Mab08, Mab09, Ber16} and
\cite{CDS15a, CDS15b, CDS15c, Tia15}. K-stability is stronger than K-polystability, and K-polystability
is stronger than K-semistability. Hence K-semistable $\bQ$-Fano varieties are interesting for both differential geometers and algebraic geometers.

It also turned out that K\"ahler--Einstein metrics and K-stability play crucial roles for construction of nice moduli spaces of certain $\bQ$-Fano varieties. For example,  compact moduli spaces of
smoothable K\"ahler--Einstein $\bQ$-Fano varieties have been constructed (see \cite{OSS16} for dimension two case and \cite{LWX14, SSY16, Oda15} for higher dimensional case).
In order to consider the moduli space of certain (singular) $\bQ$-Fano varieties, the first step is to show the boundedness property, which is the motivation of this paper.  We show the boundedness of  K-semistable $\bQ$-Fano varieties of fixed dimension with anti-canonical degrees bounded from below, which gives an affirmative answer to a question asked by Yuchen Liu during the AIM workshop ``Stability and moduli spaces" in January 2017.

\begin{thm}\label{bdd2}
Fix a positive integer $d$ and a real number $\delta>0$. Then the set of $d$-dimensional K-semistable $\bQ$-Fano varieties $X$ with $(-K_X)^d>\delta$ forms a bounded family. 
\end{thm}

Note that the assumption that $(-K_X)^d$ is bounded from below is necessary, by Example \ref{ex}(2) later. 

As mentioned before, one might have further applications of Theorem \ref{bdd2} such as constructing moduli spaces of  $d$-dimensional K-semistable $\bQ$-Fano varieties with bounded anti-canonical degrees. An interesting corollary of Theorem \ref{bdd2} is the discreteness of the anti-canonical degrees of K-semistable $\bQ$-Fano varieties.
\begin{cor}\label{vol finite}
Fix a positive integer $d$. Then the set
of $(-K_X)^d$ for 
$d$-dimensional K-semistable $\bQ$-Fano varieties $X$ is finite away from $0$.
\end{cor}
Here a set $\PP$ of positive real numbers is {\it finite away from $0$} if for any $\delta>0$, $\PP\cap (\delta, \infty)$ is a finite set. We remark that Corollary \ref{vol finite} might be related to the conjectural discreteness of minimal normalized volumes of klt singularities, cf. \cite[Question 4.3]{LX17}.

The idea of proof of Theorem \ref{bdd2} comes from birational geometry. According to Minimal Model Program, $\bQ$-Fano varieties form a fundamental class in birational geometry, and the boundedness property for $\bQ$-Fano varieties are also interesting from the point view of birational geometry. For example, Koll\'ar, Miyaoka, and Mori \cite{KMM92} proved that smooth Fano varieties form a bounded family. The most celebrated progress recently is the proof of Borisov--Alexeev--Borisov Conjecture due to Birkar \cite{Bir16a, Bir16b}, which says that given a positive integer $d$ and a real number $\epsilon>0$, the set of $\epsilon$-lc $\bQ$-Fano varieties of dimension $d$ forms a bounded family.

In this paper, inspired by Birkar's work, in order to show Theorem \ref{bdd2},
we show the following theorem.
\begin{thm}\label{bdd1}
Fix a positive integer $d$ and a real number $\delta>0$. Then the set of $d$-dimensional $\bQ$-Fano varieties $X$ with $(-K_X)^d>\delta$ and $\alpha(X)>\delta$ forms a bounded family. 
\end{thm} 
Here $\alpha(X)$ is the {\it alpha-invariant}  of $X$ defined by Tian \cite{Tia87} (see also  \cite{Dem08}) in order to investigate the existence of K\"ahler--Einstein metrics on Fano manifolds. Recall that Fujita and Odaka  \cite[Theorem 3.5]{FO} proved that the alpha-invariant of a K-semistable $\bQ$-Fano variety of dimension $d$ is always not less than $1/(d+1)$, so Theorem \ref{bdd1} implies Theorem \ref{bdd2} naturally. The advantage to consider Theorem \ref{bdd1} is that we can then apply methods from birational geometry, instead of dealing with K-semistable $\bQ$-Fano varieties.

The point of Theorem \ref{bdd1} is that we replace the $\epsilon$-lc condition in Borisov--Alexeev--Borisov Conjecture by the condition on lower bound of anti-canonical degrees and alpha-invariants, which are global invariants. 

We remark that if one take $\delta=1$, then Theorem \ref{bdd1} is a consequence of \cite[Theorem 1.3]{Bir16a}, which says that the set of {\it exceptional} $\bQ$-Fano varieties (i.e., $\bQ$-Fano varieties $X$ with $\alpha(X)>1$) of fixed dimension forms a bounded family. Note that in this case we even do not need to assume $(-K_X)^d$ is bounded from below. But in general we need to assume both $(-K_X)^d$ and $\alpha(X)$ are bounded from below, by the following examples.
\begin{example} \label{ex}Fix a positive integer $d$.
\begin{enumerate}
\item Consider the weighted projective space $X_n=\bP(1^{d}, n)$ which is a $\bQ$-Fano variety of dimension $d$ with $(-K_{X_n})^d=(n+d)^d/n>1$, but it is clear that $\{X_n\}$ does not form a bounded family.
\item Consider $Y_{8n+4}\subset \bP(2,2n+1,2n+1,4n+1)$, a general weighted hypersurface of degree $8n+4$, which is a $\bQ$-Fano variety of dimension $2$ with $\alpha(Y_{8n+4})=1$ (see \cite[Corollary 1.12]{CPS10} or \cite{JK01}), but it is clear that $\{Y_{8n+4}\}$ does not form a bounded family. For more interesting examples of $\bQ$-Fano varieties with $\alpha\geq 1$, we refer to \cite{CPS10, CS13} in dimension $2$ and \cite{CS11, CS14} in higher dimensions. Note that all examples with $\alpha\geq 1$ are K-semistable (in fact, K-stable) by  \cite[Theorem 1.4]{OS12} (or \cite{Tia87}).
\end{enumerate}
\end{example}

By \cite[Proposition 7.13]{Bir16a} or \cite[Theorem 2.15]{Bir16b},  Theorem \ref{bdd1} is a consequence of the following theorem.
\begin{thm}\label{eff bir}
Fix a positive integer $d$ and a real number $\delta>0$. Then there exists a positive integer $m$ depending only on $d$ and $\delta$  such that if $X$ is a $d$-dimensional $\bQ$-Fano variety with $(-K_X)^d>\delta$ and $\alpha(X)>\delta$, then $|-mK_X|$ defines a birational map.
\end{thm}

To show Theorem \ref{eff bir}, our main idea is to establish an inequality expressed in terms of the volume of $-K_X|_G$ on a covering family of subvarieties $G$ of $X$ and $(-K_X)^d$, $\alpha(X)$, see Lemma \ref{K.G}.

As a variation of Theorem \ref{bdd1}, we can also show the following theorem.
\begin{thm}\label{bdd3}
Fix a positive integer $d$ and a real number $\theta>0$. Then the set of $d$-dimensional $\bQ$-Fano varieties $X$ with $\alpha(X)^d\cdot (-K_X)^d>\theta$ forms a bounded family. 
\end{thm}

Logically, Theorem \ref{bdd1} is implied by Theorem \ref{bdd3}. But we will show Theorem \ref{bdd1} first in order to make the explanation more clear. 
\begin{remark}
Note that the invariant $\alpha(X)^d\cdot (-K_X)^d$ appears naturally in birational geometry, see for example \cite[Theorem 6.7.1]{SOP}. It is not clear whether we can replace  $\alpha(X)^d\cdot (-K_X)^d$ in Theorem \ref{bdd3} by $\alpha(X)^{d'}\cdot (-K_X)^d$ for some positive real number $d'<d.$ At least $d'\leq d-1$ is not sufficient to conclude the boundedness. For example, in Example \ref{ex}(1), $(-K_{X_n})^d=(n+d)^d/n$ and $\alpha(X_n)=1/(n+d)$ (for computation of alpha-invariants of toric varieties, see \cite[6.3]{Amb16}), hence $\alpha({X_n})^{d-1}\cdot (-K_{X_n})^d>1$.
\end{remark}
\begin{remark}
We remark that the proof of both Theorems \ref{bdd1} and \ref{bdd3} works under weaker assumption that $X$ is a {\it weak $\bQ$-Fano variety} (i.e., $X$ has at most klt singularities and $-K_X$ is nef and big), see also Remark \ref{remark1}. But it is not clear yet whether the log Fano pair versions hold or not.
\end{remark}

{\it Acknowledgment.} The topic of this paper was brought to the author by Yuchen Liu. The author would like to thank Yuchen Liu for inspiration and fruitful discussions.
The author is grateful to Jingjun Han, Yusuke Nakamura, and Taro Sano for discussions and to Kento Fujita, Yujiro Kawamata, Ivan Cheltsov, Wenhao Ou, and the referee(s) for valuable comments and suggestions. Part of this paper was written during the author enjoyed the workshops  ``NCTS Workshop on Singularities, Linear Systems, and Fano Varieties" at National Center for Theoretical Sciences  and   ``BICMR-Tokyo Algebraic Geometry Workshop" at Beijing International Center for Mathematical Research. The author is grateful for the hospitality and support of these institutes.

\section{Preliminaries}\subsection{Notation and conventions}\label{section notation}
 We adopt the standard notation and definitions in \cite{KMM} and \cite{KM}, and will freely use them.

A {\it pair} $(X, B)$ consists of  a normal projective variety $X$ and an effective
$\mathbb{R}$-divisor $B$ on $X$ such that
$K_X+B$ is $\mathbb{R}$-Cartier.   


Let $f: Y\rightarrow X$ be a log
resolution of the pair $(X, B)$, write
$$
K_Y =f^*(K_X+B)+\sum a_iF_i,
$$
where $\{F_i\}$ are distinct prime divisors. 
Take a real number $\epsilon\geq 0$. The pair $(X,B)$ is called
\begin{itemize}
\item[(a)] \emph{kawamata log terminal} (\emph{klt},
for short) if $a_i> -1$ for all $i$;

\item[(b)] \emph{log canonical} (\emph{lc}, for
short) if $a_i\geq  -1$ for all $i$;

\item[(c)] \emph{$\epsilon$-log canonical} (\emph{$\epsilon$-lc}, for
short) if $a_i\geq  -1+\epsilon$ for all $i$.

\end{itemize}
Usually we write $X$ instead of $(X,0)$ in the case $B=0$.
  
   $F_i$ is called a {\it non-klt place} of $(X, B)$  if $a_i\leq -1$.
A subvariety $V\subset X$ is called a {\it non-klt center} of $(X, B)$ if it is the image of a non-klt place. 

Let $(X, B)$ be  an lc pair and $D\geq 0$ be a $\bR$-Cartier $\bR$-divisor. The
{\it log canonical threshold} of $D$ with respect to $(X, B)$ is defined by
$$\lct(X, B; D) = \sup\{t\geq 0 \mid (X, B+ tD) \text{ is lc}\}.$$

If  $X$ is a $\bQ$-Fano variety,  the {\it  alpha-invariant} of $X$ is defined by 
$$
\alpha(X)=\inf\{\lct(X;D)\mid D\sim_\bQ-K_X, D\geq 0\}.
$$

A collection of varieties $\{X_t\}_{t\in T}$ is
said to be \emph{bounded}  if there exists $h:\mathcal{X}\rightarrow
S$ a projective morphism between schemes of finite type such that 
each $X_t$  is isomorphic to $\mathcal{X}_s$ for some $s\in S$.

\subsection{Volumes}

Let $X$ be a $d$-dimensional normal projective variety  and $D$ be a Cartier divisor on $X$. The {\it volume} of $D$ is the real number
$$
{\vol_X}(D)=\limsup_{m\rightarrow \infty}\frac{h^0(X,\OO_X(mD))}{m^d/d!}.
$$
Note that the limsup is actually a limit. Moreover by the homogenous property and continuity of  volumes, we can extend the definition to $\bR$-Cartier $\bR$-divisors. Note that if $D$ is a nef $\bR$-Cartier $\bR$-divisor, then $\vol_X(D)=D^d$. 

For more background on volumes, see \cite[2.2.C, 11.4.A]{Positivity2}. 

\subsection{Potentially birational divisors}
 Let $X$ be a normal projective variety and $D$ be a big $\bQ$-Cartier $\bQ$-divisor on $X$. We say that $D$ is {\it potentially birational} (see \cite[Definition 3.5.3]{ACC}) if for any two general points $x$ and $y$ of $X$, possibly switching $x$ and $y$, we can find an effective $\bQ$-divisor $\Delta \sim_\bQ (1-\epsilon)D$ for some $0<\epsilon<1$ such that $(X,\Delta)$ is not klt at $y$ but $(X,\Delta)$ is lc at $x$ and $\{x\}$ is a non-klt center. Note that if $D$ is potentially birational, then $|K_X + \roundup{D}|$ defines a birational map (\cite[Lemma
2.3.4]{HMX13}).
\subsection{Non-klt centers}
We recall the following proposition in  \cite{Bir16a} which is proved by standard techniques for constructing families of non-klt centers, see e.g. \cite{SOP, ACC, Bir16a}.
\begin{prop}[{cf. \cite[2.31(2), page 22]{Bir16a}}]\label{nonklt}
Let $X$ be a normal projective variety of dimension $d$ and $D$, $A$ be two ample $\bQ$-divisors. Assume $D^d > (2d)^d$. 

Then there is a bounded family of subvarieties of $X$ such that for two general points $x,y$ in $X$, there is a member $G$ of the family and an effective $\bQ$-divisor $\Delta\sim_\bQ D+(d-1)A$ such that
 \begin{itemize}
\item $(X,\Delta)$ is lc near $x$ with a unique non-klt place whose center contains $x$, that center is $G$, 
\item $(X, \Delta)$ is not klt at $y$, and 
\item either $\dim G= 0$ or $(A|_G)^{\dim G} \leq  d^d$.
\end{itemize}
\end{prop}
\subsection{Birkar's results}
We recall several theorems from \cite{Bir16a}. The following theorem provides a criterion of boundedness of certain $\bQ$-Fano varieties, which is one of the key ingredients of \cite{Bir16a, Bir16b}.
\begin{thm}[{\cite[Proposition 7.13]{Bir16a}}]\label{thm bir1}
Let $d,m,v$ be positive integers and $t_l$ be a sequence of positive real numbers. Let $\PP$ be the set of projective varieties $X$ such that
\begin{itemize}
\item $X$ is a $\bQ$-Fano variety of dimension $d$, 
\item $K_X$ has an $m$-complement,
\item $| -mK_X |$ defines a birational map,
\item $(-K_X)^d \leq  v$, and
\item for any $l \in \bN$ and any $L \in | - lK_X |$, the pair $(X, t_l L)$ is klt. 
\end{itemize}
Then $\PP$ is a bounded family.
\end{thm}
Here $K_X$ has an {\it $m$-complement} means that there exists an effective divisor $M\sim -mK_X$, such that $(X, \frac{1}{m}M)$ is lc. For definition of complements in general setting, we refer to \cite{Bir16a}. The boundedness of complements is proved by Birkar as the following theorem.
\begin{thm}[{\cite[Theorem 1.1]{Bir16a}}]\label{thm bir2}Let $d$ be a positive integer. Then there exists a positive integer $n$ depending only on $d$ such that if $X$ is a $\bQ$-Fano variety of dimension $d$, then $K_X$ has an $n$-complement. 
\end{thm}

Recall that a $\bQ$-Fano variety $X$ is {\it exceptional} if $(X, D)$ is klt for any effective $\bQ$-divisor $D\sim_\bQ -K_X$. This is equivalent to say that $\alpha(X)>1$, because, if $\alpha(X)>1$ then clearly $X$ is exceptional by the definition (note that we only need this direction of the implication in this paper); on the other hand, if $X$ is {exceptional}, then it is easy to see that $\alpha(X)\geq 1$, but one can use \cite[Theorem 1.5]{Bir16b} to exclude the case $\alpha(X)=1$.

 Birkar proved the boundedness of exceptional $\bQ$-Fano varieties.
\begin{thm}[{\cite[Theorem 1.3]{Bir16a}}]\label{thm bir3}Let $d$ be a positive integer.  Then the set of exceptional $\bQ$-Fano varieties of dimension $d$ forms a bounded family.
\end{thm}

\begin{remark}\label{remark1}
All theorems in this subsection hold for weak $\bQ$-Fano varieties.
\end{remark}
\section{Proof of the theorems}
The idea of proof of Theorem \ref{eff bir} is to construct isolated non-klt centers by $-K_X$, that is, for a general point $x\in X$, we need to construct an effective $\bQ$-divisor $\Delta\sim_\bQ -mK_X$ where $m$ is fixed, so that $(X, \Delta)$ has an isolated non-klt center at $x$, see e.g. \cite{AS95, HMX13, Bir16a}. From the lower bound of $(-K_X)^d$, it is easy to construct some non-klt center $G$ containing $x$. In order to cut down the dimension of $G$, we need to bound the volume of $-K_X|_G$. The main point of this paper is to show that the volume of $-K_X|_G$ is bounded from below by an expression in terms of $(-K_X)^d$ and $\alpha(X)$, as
the following lemma.
\begin{lem}\label{K.G}Fix two positive integers $d>k$. 
Let $X$ be a $\bQ$-Fano variety of dimension $d$. Assume there is a contraction  $f: Y\to T$  of projective varieties with a surjective morphism $\phi:Y\to X$. Assume that a general fiber $F$ of $f$ is of dimension $k$ and is mapped birationally onto its image $G$ in $X$, $\phi$ is smooth over $\phi(\eta_F)$, and $\phi(\eta_F)$ is a smooth point of $X$. Then
$$(-K_X)^k\cdot G\geq \frac{\alpha(X)^{d-k}}{\binom{d}{k}(d-k)^{d-k}}(-K_X)^d.$$
\end{lem}
\begin{proof}
Taking normalizations and resolutions, we may assume $Y$ and $T$ are smooth. We may pick a general fiber $F$ over $t\in T$ such that $f$ is smooth over $t$. Cutting by general smooth hyperplane sections of $T$ containing $t$, we may assume $\phi$ is generically finite, here note that all the assumptions are preserved according to \cite[Lemma 2.28]{Bir16a}. In particular, $\dim Y=d$.

Fix any rational number $l$ such that 
$$l>\sqrt[d-k]{\binom{d}{k}\frac{(-K_X)^k\cdot G}{(-K_X)^d}},$$
take an integer $m$ such that $ml$ is an integer and $-lmK_X$ is Cartier.

Note that there is a natural injection $\OO_X /\III^{\langle m\rangle}_G \to \phi_*(\OO_Y /\III^m_F)$ by comparing the order of local regular functions since $\phi$ is \'etale over the generic point of $G$. Here $\III_F$ denotes the ideal sheaf of $F$ and  $\III^{\langle m\rangle}_G$ is the ideal of regular
functions vanishing along a general point of $G$ to order at least $m$.
By projection formula, this implies that 
\begin{align*}
{}&h^0(X, \OO_X(-lmK_X)\otimes  \OO_X /\III^{\langle m\rangle}_G)\\\leq {}& h^0(X, \OO_X(-lmK_X)\otimes \phi_*(\OO_Y /\III^m_F))\\
= {}&h^0(Y, \phi^*\OO_X(-lmK_X)\otimes \OO_Y /\III^m_F).
\end{align*}
On the other hand, since $F$ is a general fiber of $f$ and $Y$ is smooth, the conormal sheaf of $F$ is trivial, that is, $\III_F/\III_F^2\simeq \OO_F^{\oplus(d-k)}$, also we have $$\III_F^{i-1}/\III_F^{i}\simeq S^{i-1}(\III_F/\III_F^2)\simeq \OO_F^{\oplus\binom{i+d-k-2}{d-k-1}}$$ for $i\geq 1$ (see \cite[II. Theorem 8.24]{H}).
Hence 
\begin{align*}
{}&h^0(Y, \phi^*\OO_X(-lmK_X)\otimes \OO_Y /\III^m_F)\\
\leq {}& \sum_{i=1}^{m}h^0(Y, \phi^*\OO_X(-lmK_X)\otimes \III_F^{i-1} /\III_F^i)\\
={}&\sum_{i=1}^{m}\binom{i+d-k-2}{d-k-1}h^0(F, \phi^*\OO_X(-lmK_X)|_F)\\
={}&\binom{m+d-k-1}{d-k}h^0(F, \phi^*\OO_X(-lmK_X)|_F).
\end{align*}
Here for the last equality, we use the formula that for positive integers $a$ and $b$,
$$\sum_{i=1}^b\binom{i+a-1}{a}=\binom{a+b}{a+1}.
$$
This can be shown by induction on $b$.

By the exact sequence
$$
0\to \OO_X(-lmK_X)\otimes \III^{\langle m\rangle}_G \to\OO_X(-lmK_X)\to \OO_X(-lmK_X)\otimes  \OO_X /\III^{\langle m\rangle}_G\to 0,
$$
we have
\begin{align*}
{}&h^0( X, \OO_X(-lmK_X)\otimes \III^{\langle m\rangle}_G )\\
\geq {}&h^0(X, \OO_X(-lmK_X))-h^0(X,  \OO_X(-lmK_X)\otimes  \OO_X /\III^{\langle m\rangle}_G)\\
\geq {}&h^0(X, \OO_X(-lmK_X))-\binom{m+d-k-1}{d-k}h^0(F, \phi^*\OO_X(-lmK_X)|_F).
\end{align*}
Note that  by definition of volumes,
$$
\lim_{m\to \infty} \frac{d!}{m^d}h^0(X, \OO_X(-lmK_X))=\vol_X(\OO_X(-lK_X))=(-lK_X)^d,
$$
and 
\begin{align*}
{}&\lim_{m\to \infty} \frac{d!}{m^d}\binom{m+d-k-1}{d-k}h^0(F, \phi^*\OO_X(-lmK_X)|_F)\\
={}&\lim_{m\to \infty} \frac{d!}{m^d}\cdot\frac{(m+d-k-1)\cdots (m+1)m}{(d-k)!}h^0(F, \phi^*\OO_X(-lmK_X)|_F)\\
={}&\lim_{m\to \infty} \frac{d!}{m^k(d-k)!}h^0(F, \phi^*\OO_X(-lmK_X)|_F)\\
={}&\binom{d}{k}\vol_F(\phi^*\OO_X(-lK_X)|_F)\\
={}&\binom{d}{k}(-\phi^*(lK_X)|_F)^k\\
={}&\binom{d}{k}(-\phi^*(lK_X))^k\cdot F\\
={}&\binom{d}{k}(-lK_X)^k\cdot G.
\end{align*}
Here for the last step we use the fact that $F\to G$ is birational (cf. \cite[Proposition 1.35(6)]{KM}).
Note that by choice of $l$,
$$
(-lK_X)^d=l^d(-K_X)^d>l^k \binom{d}{k}\frac{(-K_X)^k\cdot G}{(-K_X)^d} (-K_X)^d=\binom{d}{k}(-lK_X)^k\cdot G.
$$
Hence 
$$
h^0( X, \OO_X(-lmK_X)\otimes \III^{\langle m\rangle}_G )>0
$$
for $m$ sufficiently large.
This implies that there exists an effective $\bQ$-divisor $D\sim_\bQ -lK_X$ such that $\mult_G D\geq 1.$ In particular, $(X, (d-k)D)$ is not klt along $G$ as $G$ interests the smooth locus of $X$ (cf. \cite[Lemma 2.29]{KM}). Hence
$$
(d-k)l\geq \lct\bigg(X; \frac{1}{l}D\bigg)\geq  \alpha(X).
$$
Since we can take $l$ to be an arbitrary rational number such that 
$$l>\sqrt[d-k]{\binom{d}{k}\frac{(-K_X)^k\cdot G}{(-K_X)^d}},$$
it follows that
$$
(d-k)\sqrt[d-k]{\binom{d}{k}\frac{(-K_X)^k\cdot G}{(-K_X)^d}}\geq  \alpha(X),
$$
that is,
$$
(-K_X)^k\cdot G\geq \frac{ \alpha(X)^{d-k}}{\binom{d}{k}(d-k)^{d-k}}(-K_X)^d.
$$
The proof is completed.
\end{proof}

\begin{proof}[Proof of Theorem \ref{eff bir}]
Take a $d$-dimensional $\bQ$-Fano variety $X$ with $(-K_X)^d>\delta$ and $\alpha(X)>\delta$.  
Take 
\begin{align*}
q_0{}&=\frac{2d}{\sqrt[d]{\delta}},\\
p_0{}&=\max_{1\leq k \leq d-1}\left\{\sqrt[k]{\frac{\binom{d}{k}(d-k)^{d-k}d^d}{\delta^{d-k+1}}}\right\}.
\end{align*}
Take roundups $q=\roundup{q_0}$ and $p=\roundup{p_0}$.
 By definition, $(-qK_X)^d>(2d)^d$.

Applying Proposition \ref{nonklt} for $D=-qK_X$ and $A=-pK_X$, then there is a bounded family of subvarieties of $X$ such that for two general points $x,y$ in $X$, there is a member $G$ of the family and an effective $\bQ$-divisor $\Delta\sim_\bQ D+(d-1)A$ such that $(X,\Delta)$ is lc near $x$ with a unique non-klt place whose center contains $x$, that center is $G$, and $(X, \Delta)$ is not klt at $y$, and either $\dim G= 0$ or $(A|_G)^{\dim G} \leq  d^d$. We will show that the latter case will never happen, that is, $\dim G= 0$ always holds for general $x,y$.

 Note that since $x,y$ are general, we can assume $G$ is a general member of the family. Recall from \cite[2.27]{Bir16a} that this means the family is given by finitely many morphisms $V_j \to T_j$ of projective varieties with surjective morphisms $V_j \to X$ such that each $G$ is a general fiber of one of these morphisms.
 If $G$ is a general fiber of some $V_j \to T_j$ of dimension $k>0$, as $V_j \to T_j$ is constructed from the Hilbert scheme of subvarieties (cf. \cite[2.27, 2.31]{Bir16a}), it satisfies the assumptions of Lemma \ref{K.G}, and then by applying Lemma \ref{K.G} to $V_j \to T_j$ and $G$, we get
$$
(-K_X)^k\cdot G\geq  \frac{ \alpha(X)^{d-k}}{\binom{d}{k}(d-k)^{d-k}}(-K_X)^d> \frac{\delta^{d-k+1}}{\binom{d}{k}(d-k)^{d-k}}.
$$
In particular, by the definition of $p$, 
$$
(A|_G)^{\dim G}=(-pK_X)^k\cdot G> d^d.
$$
This is a contradiction.

Hence $\dim G=0$  for general $x,y$,
 that is, $G=\{x\}$. Recall our construction, this means that
 for any two general points $x,y \in X$ we can choose an effective $\bQ$-divisor $\Delta\sim_\bQ D+(d-1)A= -(q+p(d-1))K_X$ so that $(X,\Delta)$ is lc near $x$ with a unique non-klt place whose center is  $\{x\}$, and $(X,\Delta)$ is not klt at $y$. Hence $-({q+p(d-1)}+1)K_X$ is potentially birational and hence $|K_X-({q+p(d-1)}+1)K_X|$ defines a birational map by \cite[Lemma
2.3.4]{HMX13}. We may take $m={q+p(d-1)}$.
\end{proof}

\begin{proof}[Proof of Theorem \ref{bdd1}]By Theorem \ref{thm bir1}, it suffices to show that there exist  positive integers $m,v$ and  a sequence of positive real numbers $t_l$ depending only on $d$ and $\delta$ such that if $X$ is a $d$-dimensional $\bQ$-Fano variety with $(-K_X)^d>\delta$ and $\alpha(X)>\delta$, then the conditions in  Theorem \ref{thm bir1} are satisfied, that is, 
\begin{enumerate}
\item $K_X$ has an $m$-complement,
\item $| -mK_X |$ defines a birational map,
\item $(-K_X)^d \leq  v$, and
\item for any $l \in \bN$ and any $L \in | - lK_X |$, the pair $(X, t_l L)$ is klt. 
\end{enumerate}

Firstly, by Theorems \ref{thm bir2} and \ref{eff bir}, there exists a positive integer $m$ depending only on $d$ and $\delta$ such that $K_X$ has an $m$-complement and $| -mK_X |$ defines a birational map.

Secondly, it is well-known (cf. \cite[Theorem 6.7.1]{SOP}) that $$\alpha(X)\leq \frac{d}{ \sqrt[d]{(-K_X)^d} }.$$ 
In fact, for any rational number $s$ such that $s>\frac{d}{ \sqrt[d]{(-K_X)^d}}$, we have $$(-sK_X)^d>d^d.$$
By \cite[Theorem 6.7.1]{SOP}, there exists an effective $\bQ$-divisor $B\sim_\bQ-sK_X$ such that $(X, B)$ is not lc. Hence $\alpha(X)<s$. By the arbitrarity of $s$, we conclude that $$\alpha(X)\leq \frac{d}{ \sqrt[d]{(-K_X)^d} }.$$ 
Hence
$$
(-K_X)^d\leq \frac{d^d}{\alpha(X)^d}< \frac{d^d}{\delta^d}
$$
and we may take $v=d^d/\delta^d$.

Finally, for any $l \in \bN$ and any $L \in | - lK_X |$, the pair $(X,  \frac{\delta}{l}L)$ is klt since $\alpha(X)>\delta$. We may take $t_l=\delta/l$.

In summary, by Theorem \ref{thm bir1}, the set of $d$-dimensional $\bQ$-Fano varieties $X$ with $(-K_X)^d>\delta$ and $\alpha(X)>\delta$ forms a bounded family.
\end{proof}

\begin{proof}[Proof of Theorem \ref{bdd3}]
Take a $d$-dimensional $\bQ$-Fano variety $X$ with $\alpha(X)^d\cdot (-K_X)^d>\theta$. We want to apply Theorem \ref{bdd1} in this situation, that is, it suffices to show that there exists a real number $\delta>0$ depending only on $d$ and $\theta$ such that $(-K_X)^d>\delta$ and $\alpha(X)>\delta$.

Firstly, note that if $\alpha(X)>1$, $X$ is an exceptional $\bQ$-Fano variety and hence belongs to a bounded family by Theorem \ref{thm bir3}. 

Hence from now on, we may assume that $\alpha(X)\leq 1$. In particular, 
$$
(-K_X)^d>\frac{\theta}{\alpha(X)^d}\geq \theta.
$$

Take 
\begin{align*}
q_0{}&=\frac{2d}{\sqrt[d]{\theta}},\\
p_0{}&=\max_{1\leq k \leq d-1}\left\{\sqrt[k]{\frac{\binom{d}{k}(d-k)^{d-k}d^d}{\theta}}\right\}.
\end{align*}
Take roundups $q=\roundup{q_0}$ and $p=\roundup{p_0}$.
 By definition, $$(-q\alpha(X) K_X)^d\geq \frac{(2d)^d}{\theta}\cdot \alpha(X)^d\cdot (-K_X)^d>(2d)^d.$$

Applying Proposition \ref{nonklt} for $D=-q\alpha(X)K_X$ and $A=-p\alpha(X)K_X$, then there is a bounded family of subvarieties of $X$ such that for two general points $x,y$ in $X$, there is a member $G$ of the family and an effective $\bQ$-divisor $\Delta\sim_\bQ D+(d-1)A$ such that $(X,\Delta)$ is lc near $x$ with a unique non-klt place whose center contains $x$, that center is $G$, and $(X, \Delta)$ is not klt at $y$, and either $\dim G= 0$ or $(A|_G)^{\dim G} \leq  d^d$. We will show that the latter case will never happen, that is, $\dim G= 0$ always holds for general $x,y$.

 Note that since $x,y$ are general, we can assume $G$ is a general member of the family. Recall from \cite[2.27]{Bir16a} that this means the family is given by finitely many morphisms $V_j \to T_j$ of projective varieties with surjective morphisms $V_j \to X$ such that each $G$ is a general fiber of one of these morphisms.
 If $G$ is a general fiber of some $V_j \to T_j$ of dimension $k>0$, as $V_j \to T_j$ is constructed from the Hilbert scheme of subvarieties (cf. \cite[2.27, 2.31]{Bir16a}), it satisfies the assumptions of Lemma \ref{K.G}, and then by applying Lemma \ref{K.G} to $V_j \to T_j$ and $G$, we get
$$
(-K_X)^k\cdot G\geq  \frac{ \alpha(X)^{d-k}}{\binom{d}{k}(d-k)^{d-k}}(-K_X)^d> \frac{\theta}{\binom{d}{k}(d-k)^{d-k}\alpha(X)^k}.
$$
In particular, by the definition of $p$, 
$$
(A|_G)^{\dim G}=(-p\alpha(X)K_X)^k\cdot G> d^d.
$$
This is a contradiction.

Hence $\dim G=0$  for general $x,y$,
 that is, $G=\{x\}$. Recall our construction, this means that
 for any two general points $x,y \in X$ we can choose an effective $\bQ$-divisor $\Delta\sim_\bQ D+(d-1)A= -(q+p(d-1))\alpha(X)K_X$ so that $(X,\Delta)$ is lc near $x$ with a unique non-klt place whose center is  $\{x\}$, and $(X,\Delta)$ is not klt at $y$. This means that the non-klt locus (i.e. union of all non-klt centers) $\text{Nklt}(X, \Delta)$ contains $y$ and $x$ such that $x$ is an isolated point. By Shokurov--Koll\'ar connectedness lemma (see Shokurov \cite{Sho93, Sho94} and Koll\'ar \cite[Theorem 17.4]{Kol92}), $-(K_X+\Delta)$ can not be ample. On the other hand, 
 $$
 -(K_X+\Delta)\sim_\bQ -(1-(q+p(d-1))\alpha(X))K_X.
 $$
 As $-K_X$ is ample, this implies that $1-(q+p(d-1))\alpha(X)\leq 0$, that is, $$\alpha(X)\geq \frac{1}{q+p(d-1)}. $$ Hence we may take $\delta=\min\{\theta, 1/2(q+p(d-1))\}$ and apply Theorem \ref{bdd1} to conclude Theorem \ref{bdd3}.
\end{proof}
\begin{proof}[Proof of Theorem \ref{bdd2}]Without loss of generality, we may assume $\delta<1/(d+1)$. 
For a K-semistable $\bQ$-Fano variety $X$ of dimension $d$, by \cite[Theorem 3.5]{FO}, $\alpha(X) \geq {1}/{(d+1)}>\delta$. Hence Theorem \ref{bdd2} follows immediately from Theorem \ref{bdd1}.
\end{proof}
\begin{proof}[Proof of Corollary \ref{vol finite}]
By Theorem \ref{bdd2}, the set of $d$-dimensional K-semistable $\bQ$-Fano varieties $X$ with $(-K_X)^d>\delta$ forms a bounded family, hence $(-K_X)^d$ can only take finitely many possible values for such $\bQ$-Fano varieties. 
\end{proof}


\begin{thebibliography}{KMM87}
\bibitem[Amb16]{Amb16} F. Ambro, {\it Variation of log canonical thresholds in linear systems}, Int. Math. Res. Not. IMRN 2016, no. 14, 4418--4448.
\bibitem[AS95]{AS95} U. Angehrn, Y. T. Siu, {\it Effective freeness and point separation for adjoint bundles}, Invent. Math. {\bf 122} (1995), no.2, 291--308.
\bibitem[Bir16a]{Bir16a}  C. Birkar, {\it Anti-pluricanonical systems on Fano varieties}, arXiv:1603.05765v3.
\bibitem[Bir16b]{Bir16b} C. Birkar,  {\it Singularities of linear systems and boundedness of Fano varieties}, arXiv:1609.05543.
\bibitem[Ber16]{Ber16} R. Berman, {\it K-polystability of $\bQ$-Fano varieties admitting K\"ahler--Einstein
metrics}, Invent. Math. {\bf 203} (2016), no. 3, 973--1025.

\bibitem[CS11]{CS11} I. Cheltsov, C. Shramov, {\it On exceptional quotient singularities}, Geom. Topol. {\bf 15} (2011), no. 4, 1843--1882.
\bibitem[CS13]{CS13} I. Cheltsov, C. Shramov,  {\it Del Pezzo zoo}, Exp. Math. {\bf 22} (2013), no. 3, 313--326.
\bibitem[CS14]{CS14} I. Cheltsov, C. Shramov, {\it Weakly-exceptional singularities in higher dimensions}, J. Reine Angew. Math. {\bf 689} (2014), 201--241.
\bibitem[CPS10]{CPS10} I. Cheltsov, J. Park, C. Shramov, {\it Exceptional del Pezzo hypersurfaces}, J. Geom. Anal. {\bf 20} (2010), no. 4, 787--816.
\bibitem[CDS15a]{CDS15a} X. Chen, S. Donaldson, S. Sun, {\it K\"ahler--Einstein metrics on Fano
manifolds, I: approximation of metrics with cone singularities}, J. Amer. Math.
Soc. {\bf 28} (2015), no. 1, 183--197.
\bibitem[CDS15b]{CDS15b} X. Chen, S. Donaldson, S. Sun, {\it K\"ahler--Einstein metrics on Fano
manifolds, II: limits with cone angle less than $2\pi$}, J. Amer. Math. Soc. {\bf 28}
(2015), no. 1, 199--234.
\bibitem[CDS15c]{CDS15c} X. Chen, S. Donaldson, S. Sun, {\it K\"ahler--Einstein metrics on Fano
manifolds, III: limits as cone angle approaches $2\pi$ and completion of the main
proof}, J. Amer. Math. Soc. {\bf 28} (2015), no. 1, 235--278.
\bibitem[CT08]{CT08} X. Chen, G. Tian, {\it Geometry of K\"ahler metrics and foliations by holomorphic
discs}, Publ. Math. Inst. Hautes \'Etudes Sci. {\bf 107} (2008), 1--107.


\bibitem[Dem08]{Dem08} J.-P. Demailly, {\it On Tian's invariant and log
canonical thresholds}, Appendix to ``Log canonical thresholds of smooth Fano threefolds" by I. Cheltsov, C. Shramov, Russian Math. Surveys {\bf 63} (2008), no. 5, 945--950.

\bibitem[DT92]{DT} W. Ding, G. Tian, {\it K\"ahler--Einstein metrics and the generalized Futaki
invariant}, Invent. Math. {\bf 110} (1992), no. 2, 315--335.

\bibitem[Don02]{Don02} S. Donaldson, {\it Scalar curvature and stability of toric varieties}, J. Differential
Geom. {\bf 62} (2002), no. 2, 289--349.
\bibitem[Don05]{Don05} S. Donaldson, {\it Lower bounds on the Calabi functional}, J. Differential Geom.
{\bf 70} (2005), no. 3, 453--472.


\bibitem[FO16]{FO} K. Fujita, Y. Odaka, {\it On the K-stability of Fano varieties and anticanonical
divisors}, arXiv:1602.01305, to appear in Tohoku Math. J..

\bibitem[HMX13]{HMX13} C. D. Hacon, J. M\textsuperscript{c}Kernan, C. Xu, {\it On the birational automorphisms of varieties of general type}, Ann. of Math. (2) {\bf 177} (2013), no. 3, 1077--1111.

\bibitem[HMX14]{ACC} C. D. Hacon, J. M\textsuperscript{c}Kernan, C. Xu, {\it ACC for log canonical thresholds}, Ann. of Math. (2) {\bf 180} (2014), no. 2, 523--571.
\bibitem[Har77]{H} R. Hartshorne, \emph{Algebraic geometry}, Graduate Texts in Mathematics, No. 52, Springer-Verlag, New York, 1977.

\bibitem[JK01]{JK01}J. M. Johnson,  J. Koll\'ar, {\it K\"ahler--Einstein metrics on log del Pezzo surfaces in weighted projective 3-spaces}, Ann. Inst. Fourier {\bf 51} (2001), 69--79.

\bibitem[KMM87]{KMM} Y. Kawamata, K. Matsuda, K. Matsuki, {\em Introduction
to the minimal model problem}, Algebraic geometry, Sendai, 1985,  pp. 283--360,
Adv. Stud. Pure Math., 10, North-Holland, Amsterdam, 1987.
\bibitem[Kol92]{Kol92} J. Koll\'ar, {\it Adjunction and discrepancies}, Flips and abundance for algebraic threefolds, A Summer Seminar on Algebraic Geometry (Salt Lake City, Utah, August 1991), Ast\'erisque {\bf 211} (1992), 183--192.
 \bibitem[Kol97]{SOP} J. Koll\'ar,  \emph{Singularities of pairs}, Algebraic geometry--Santa Cruz 1995, pp. 221--287, Proc. Sympos. Pure Math., 62, Part 1, Amer. Math. Soc., Providence, RI, 1997.
 
\bibitem[KMM92]{KMM92}J. Koll\'ar, Y. Miyaoka, S. Mori, {\it Rationally connectedness and boundedness of Fano manifolds}, J. Differential Geom.  {\bf 36} (1992), no. 3,  765--769.
\bibitem[KM98]{KM} J. Koll\'{a}r, S. Mori, \emph{Birational geometry of algebraic varieties},
Cambridge tracts in mathematics, 134, Cambridge University
Press, Cambridge, 1998.
\bibitem[Laz04]{Positivity2}  R. Lazarsfeld, \emph{Positivity in algebraic geometry, I, II.} Ergeb. Math. Grenzgeb, 3. Folge.
A Series of Modern Surveys in Mathematics [Results in Mathematics and Related Areas.
3rd Series. A Series of Modern Surveys in Mathematics], 48\&49. Berlin: Springer 2004.


\bibitem[LWX14]{LWX14} C. Li, X. Wang, C. Xu, {\it On proper moduli space of smoothable K\"ahler--Einstein Fano
varieties}, arXiv:1411.0761v3.

\bibitem[LX17]{LX17} Y. Liu, C. Xu, {\it K-stability of cubic threefolds},  arXiv:1706.01933.
 \bibitem[Mab08]{Mab08} T. Mabuchi, {\it K-stability of constant scalar curvature}, arXiv:0812.4903.
\bibitem[Mab09]{Mab09} T. Mabuchi, {\it A stronger concept of K-stability}, arXiv:0910.4617.

\bibitem[Oda15]{Oda15} Y. Odaka, {\it Compact moduli spaces of K\"ahler--Einstein Fano varieties}, Publ.
Res. Inst. Math. Sci. {\bf 51} (2015), no. 3, 549--565.
\bibitem[OS12]{OS12} Y. Odaka, Y. Sano, {\it Alpha invariants and K-stability of $\bQ$-Fano varieties},
Adv. Math. {\bf 229} (2012), no. 5, 2818--2834.
\bibitem[OSS16]{OSS16} Y. Odaka, C. Spotti, S. Sun, {\it Compact moduli spaces of Del Pezzo
surfaces and K\"ahler--Einstein metrics}, J. Differential Geom. {\bf 102} (2016), no. 1,
127--172.

\bibitem[Sho93]{Sho93} V. V. Shokurov, {\it 3-fold log flips}, Russian Acad. Sci. Izv. Math. {\bf 40} (1993), no. 1, 95--202.
\bibitem[Sho94]{Sho94} V. V. Shokurov, {\it An addendum to the paper ``3-fold log flips" [Russian Acad. Sci. Izv. Math. {\bf 40} (1993), no. 1, 95--202]}, Russian Acad. Sci. Izv. Math. {\bf 43} (1994), no. 3, 527--558.
\bibitem[SSY16]{SSY16} C. Spotti, S. Sun, C. Yao, {\it Existence and deformations of K\"ahler--Einstein metrics on smoothable
$\bQ$-Fano varieties}, Duke Math. J. {\bf 165} (2016), no. 16, 3043--3083.
\bibitem[Sto09]{Sto09} J. Stoppa, {\it K-stability of constant scalar curvature K\"ahler manifolds}, Adv.
Math. {\bf 221} (2009), no. 4, 1397--1408.
\bibitem[Tia87]{Tia87} G. Tian, {\it On K\"ahler--Einstein metrics on certain K\"ahler manifolds with
$C_1(M) > 0$}, Invent. Math. {\bf 89} (1987), no. 2, 225--246.
\bibitem[Tia97]{Tia97} G. Tian, {\it K\"ahler--Einstein metrics with positive scalar curvature}, Invent.
Math. {\bf 130} (1997), no. 1, 1--37.
\bibitem[Tia15]{Tia15} G. Tian, {\it K-stability and K\"ahler--Einstein metrics}, Comm. Pure Appl. Math.
{\bf 68} (2015), no. 7, 1085--1156.
\end{thebibliography}
\end{document}